\documentclass[12pt,reqno,a4paper]{amsart}
\usepackage{amssymb}
\usepackage{amsmath}

\def\a{\alpha}

\def\om{\omega}
\def\vs{\vskip .6cm}

\def\.{\cdot}

\def\n{\nabla}

\def\t{\tilde}
\def\beq{\begin{equation}}
\def\eeq{\end{equation}}
\def\bea{\begin{eqnarray*}}
\def\eea{\end{eqnarray*}}
\def\ba{\begin{array}}
\def\ea{\end{array}}

\def\k{\kappa}
\def\r{\end{proof}}
\def\res{\arrowvert}

\def\ci{{\mathcal C}^\infty}

\def \RM{\mathbb{R}}

\def \ZM{\mathbb{Z}}
\def \CM{\mathbb{C}}

\def \SM{\mathbb{S}}

\def\es{\,\lrcorner\,}

\def\id{\mathrm{id}}
\def\Lie{{\mathcal L}}
\def\U{\mathrm{U}}

\def\rk{\mathrm{rk}}

\def\z{\mathfrak z}
\def\g{\mathfrak g}
\def\t{\mathfrak t}
\def\k{\mathfrak k}
\def\Ker{\mathrm{Ker}}
\def\sp{\mathrm{Span}}
\def\T{\mathrm{T}}

\newtheorem{ede}{Definition}[section]

\newtheorem{epr}[ede]{Proposition}

\newtheorem{ath}[ede]{Theorem}

\newtheorem{elem}[ede]{Lemma}

\newtheorem{ere}[ede]{Remark}

\newtheorem{ecor}[ede]{Corollary}

\title[Homogeneous locally conformally K\"ahler manifolds]{Homogeneous locally conformally K\"ahler manifolds}
\author{Andrei Moroianu}
\author{Liviu Ornea}
\thanks{ Both authors were partially supported by LEA Math-Mode. A.M.  was partially supported by the contract ANR-10-BLAN 0105 ``Aspects Conformes
de la G\'eom\'etrie". L.O. thanks the Laboratoire de Math\'ematiques,  Universit\'e de Versailles  for hospitality during the
preparation of this work. He was also partially supported by CNCS UEFISCDI, project number PN-II
-ID-PCE-2011-3-0118.}

\address{Universit\'e de Versailles-St Quentin, Laboratoire de Math\'ematiques,
UMR 8100 du CNRS, 45 avenue des
\'Etats-Unis, 78035 Versailles, France} 
\email{andrei.moroianu@math.cnrs.fr}
\address{Univ. of Bucharest, Faculty of Mathematics,
14 Academiei str, 70109 Bucharest, Romania, and 
Institute of Mathematics ``Simion Stoilow" of the Romanian Academy, 21, Calea Grivitei str.,
010702-Bucharest, Romania.}
\email{lornea@fmi.unibuc.ro, Liviu.Ornea@imar.ro}

\begin{document}

\begin{abstract} It is known that automorphism  group $G$ of a compact homogeneous locally conformally K\"ahler manifold $M=G/H$ has at least a $1$-dimensional center. We prove that the center of $G$ is at most $2$-dimensional, and that if its dimension is $2$, then $M$ is Vaisman and isometric to a mapping torus of an isometry of a homogeneous Sasakian manifold. 

\vs


\end{abstract}

\maketitle

\section{Introduction}

Locally conformally K\"ahler (LCK in short) manifolds are Hermitian manifolds  endowed with a closed one-form $\theta$ called Lee form, such that their fundamental 2-form $\om$ satisfies the integrability condition $d\om=-2\theta\wedge\om$, see \cite{do} and the next section for precise definitions.

Such a structure becomes interesting especially on compact manifolds which are known to not be of K\"ahler type. Indeed, almost all non-K\"ahler compact surfaces are LCK, see \cite{be} and \cite{br}. Higher dimensional examples are the Hopf manifolds, \cite{ovsh}, and the Oeljeklaus-Toma manifolds, \cite{ot}. The simplest example of Hopf manifold is $H_n:=\big(\CM^n\setminus \{0\}\big)/\Gamma$, where $\Gamma\simeq\ZM$ is generated by the transformation $(z_i)\mapsto(2z_i)$, endowed with the $\Gamma$-invariant metric $|z|^{-2}\sum dz_i\otimes dz_i$, and with Lee form $\theta=d\log|z|$. More generally, Hopf manifolds are quotients  $\big(\CM^n\setminus \{0\}\big)/\Gamma$ with $\Gamma\simeq \ZM$ generated by the action of an invertible linear (not necessarily diagonal) operator $A$ with eigenvalues of absolute values strictly larger than 1. All Hopf manifolds are diffeomorphic to $\SM^1\times \SM^{2n-1}$, but their complex structure depends upon the operator $A$.

Most known examples of LCK manifolds satisfy a stronger condition: the closed one-form in the definition is parallel with respect to the Levi-Civita connection of the metric. They are called Vaisman manifolds (although I. Vaisman introduced them as ``generalized Hopf manifolds'').  When compact, Vaisman manifolds have a topology quite different from the K\"ahler ones, \cite{ovtop}, \cite{vai}.

Compact Vaisman manifolds are closely related to Sasakian manifolds: they are mapping tori over the circle with fibre a Sasakian manifold, \cite{ovst}. Among the Hopf manifolds, only the ones associated to a diagonal operator, like $H_n$, are Vaisman, \cite{ovsh}.

Unlike   compact homogeneous K\"ahler manifolds which are fully classified \cite{mat}, a structure theorem for compact homogeneous LCK manifolds is still lacking and only informations about particular classes are available. For example, combining \cite{ovst} with \cite{vai} one easily proves that compact homogeneous Vaisman manifolds are mapping tori over the circle with fibre a compact homogeneous Sasakian manifold (and these ones are total spaces of Boothby-Wang fibrations over compact homogeneous K\"ahler manifolds). Also, locally homogeneous LCK manifolds are treated in \cite{haka}.

In this note, we discuss compact homogeneous LCK manifolds according to the dimension of the center of their group of holomorphic isometries and find that if this dimension is $2$, then the manifold is Vaisman. 

The proof goes roughly as follows. If the center of the automorphism group of $M$ has dimension at least 2, one can find a holomorphic Killing vector field $\xi$ on which the Lee form $\theta$ vanishes identically. Using the compactness and homogeneity of $M$ one can show that up to a constant factor $\theta$ equals the metric dual of $J\xi$ (in particular this shows that the dimension of the center is exactly 2). The orbits of $\xi$ are closed circles of constant length so the orbit space is a smooth Riemannian manifold $N$ such that the projection $M\to N$ is a Riemannian submersion. The Lee form projects to a closed 1-form on $N$. Each integral distribution of its kernel turn out to be a homogeneous K\"ahler manifold $P$. Moreover, the second central Killing vector field on $M$ defines a holomorphic vector field on $P$, which by homogeneity has to be parallel. Translating this back to $M$ shows that the Lee form $\theta$ is parallel.

%\begin{ere} {\rm 
%We shall always suppose that the manifolds we consider are not globally conformally K\"ahler, \emph{i.e.} the Lee form is never  exact. This is especially important on compact manifolds, where LCK and K\"ahler structures impose completely different topologies. Moreover, we assume that the Lee form has no singularities. }
%\end{ere}
{\sc Acknowledgment.} We would like to thank Paul Gauduchon for many enlightening discussions during the preparation of this work.

\section{Preliminaries}

On every Riemannian manifold $(M,g)$ the metric $g$ defines isomorphisms inverse to each other $\T M\ni X\mapsto X^\flat\in \T ^*M$ and $\T ^*M\ni\alpha\mapsto \alpha^\sharp\in \T M$ by 
$$X^\flat(Y):= g(X,Y), \qquad g(\alpha^\sharp,Y):=\alpha(Y)$$
for all $Y\in \T M$. If $J$ is an almost Hermitian structure on $M$ ({\em i.e.} a skew-symmetric endomorphism of $\T M$ of square $-\id$), we also denote by $J$ the complex structure on $\T ^*M$ induced by the above isomorphisms:
$$J(X^\flat):=(JX)^\flat.$$

An almost Hermitian manifold $(M,g,J)$ is called Hermitian if $J$ is integrable, {\em i.e.} if the Nijenhuis tensor of $J$ defined by
$$N^J(X,Y):=[X,Y]+J[X,JY]+J[JX,Y]-[JX,JY],\qquad\forall\ X,Y\in\ci(\T M)$$
vanishes. Since $N^J(X,Y)=J(\Lie_X J)Y-(\Lie_{JX}J)Y$ for all vector fields $X,Y$, it follows that $J$ is integrable if and only if 
\beq\label{int}J\Lie_X J=\Lie_{JX}J,\qquad\forall\ X\in\ci(\T M).
\eeq
An almost Hermitian manifold $(M,g,J)$ is called K\"ahler if $J$ is parallel with respect to the Levi-Civita connection $\nabla$ of $g$. This is equivalent to the fact that $J$ is integrable and the fundamental 2-form $\om:=g(J\cdot,\cdot)$ is closed.

A vector field $X$ on an almost Hermitian manifold $(M,g,J)$ is called {\em holomorphic} if $\Lie_XJ=0$. By \eqref{int}, if $J$ is integrable then $X$ is holomorphic if and only if $JX$ is holomorphic. The following well known result will be used in the sequel and hence we include a proof.

\begin{elem}\label{hv} A vector field $X$ on a K\"ahler manifold $(M,g,J)$ is holomorphic if and only if $\nabla_{JY}X=J(\nabla_YX)$ for any vector field $Y$ on $M$.
\end{elem}
\begin{proof} We have
\begin{equation*}
\begin{split}(\Lie_XJ)Y&=[X,JY]-J[X,Y]=\nabla_X(JY)-\nabla_{JY}X-J(\nabla_XY)+J(\nabla_YX)\\
&=-\nabla_{JY}X+J(\nabla_YX).\end{split}
\end{equation*} 
\end{proof}

A Hermitian manifold $(M,g,J)$ is called locally conformally K\"ahler (in short LCK) if the fundamental 2-form $\om:=g(J\cdot,\cdot)$ satisfies
$$d\om=-2\theta\wedge\om$$
for some closed 1-form on $M$ called the Lee form\footnote{The other convention used for this definition is $d\om=\theta\wedge\om$, but then \eqref{lck} looks more complicated.}. Since $J$ is integrable, this readily implies 
\beq\label{lck}
\nabla_X\om=\theta\wedge JX^\flat+J\theta\wedge X^\flat,\qquad\forall\ X\in \T M.
\eeq

A LCK manifold $(M,g,J)$ is called Vaisman if the Lee form $\theta$ is parallel with respect to the Levi-Civita connection of $g$.

\section{Homogeneous LCK manifolds}

Let now $(M,g,J,\omega,\theta)$ be a compact homogeneous LCK manifold. This means that we assume the existence of a compact, connected  Lie group $G$ acting effectively and transitively on $M$ preserving  both the metric $g$ \emph{and} the complex structure $J$. Consequently, $\omega$ and $\theta$ are also preserved. In particular, the length of $\theta$ is constant. We will assume from now on that $M$ is not K\"ahler, that is $\theta\ne 0$. By a constant rescaling of the metric one can assume that $\theta$ has unit length.

Let $H$ be the isotropy subgroup of the action of $G$ at some point of $M$ and write $M=G/H$. 

\begin{ere}{\rm We may suppose that $H$ is connected. Indeed, if $H$ is not connected, we can take $H_0$, its connected component of the identity and work with $G/H_0$ which finitely covers $G/H$.}
\end{ere}

No complete classification of homogeneous LCK manifolds is available up to now. A natural example is the diagonal Hopf manifold which can be described as $H_n=\big({\SM}^1\times \U(n)\big)/\U(n-1)$, biholomorphic to the Vaisman manifold $\SM^1\times \SM^{2n-1}$ (with the LCK structure defined in the introduction). On the other hand, the Inoue surfaces $S_M$ and their generalizations, the Oeljeklaus-Toma manifolds, are solvmanifolds and, respectively, homogeneous manifolds, but are not LCK homogeneous (their group of biholomorphisms is zero-dimensional, see \cite{be}, \cite{kas}). 

Any element $a$ of the Lie algebra $\g=L(G)$ induces a fundamental vector field $X^a$ on $M$ by the formula
$$X^a_x=\frac{d}{dt}\big\arrowvert_0 \exp(ta)\cdot x,\qquad x\in M.$$
Since its flow is made by left translations with elements of $G$, $X^a$ is a Killing field on $M$. Note that $X^a$ is the projection on $M$ of the {\em right-invariant} vector field on $G$ induced by $a\in\g$, so the bracket of the vector fields $X^a$ and $X^b$ on $M$ is the opposite of the bracket of $a$ and $b$ in $\g$: $[X^a,X^b]=-X^{[a,b]}$. 

Keeping this in mind, we shall identify from now on the elements of $\g$ with the fundamental fields they induce on $M$ and hence denote them as $X,Y$ etc.

Let $\z$ be the center of $\g$. Our results depend upon the dimension of the center. In the following, we discuss the possibilities that can occur. To begin with, one easily proves: 

\begin{elem}{\rm (\cite{paul})}\label{lemap}
For any $X\in\g$, $\theta(X)$ is constant on $M$ and $\theta\perp [\g,\g]$. In particular, $\dim \z\geq 1$.
\end{elem}  
\begin{proof} Since $\theta$ is $G$-invariant and closed, the Cartan formula yields for every $X\in \g$:
$$0=\Lie_X\theta=d(X\es\theta)=d(\theta(X)).$$
Thus $\theta(X)$ is constant on $M$. Using this, we obtain that for every $X,Y\in\g$:
$$0=d\theta(X,Y)=X(\theta(Y))-Y(\theta(X))-\theta([X,Y])=-\theta([X,Y]),$$
thus showing that $\theta(Z)$ vanishes identically for every $Z\in [\g,\g]$. 
\end{proof}

A key point of our study is the following:

\begin{elem} \label{jxi} If $\dim \z\geq 2$, then $J\theta^\sharp\in\z$.
\end{elem}
\begin{proof}
For any  $\xi \in\z$, the functions $\theta(\xi)$ and $\theta(J\xi)$ are $G$-invariant, so they are constant on $M$. From our hypothesis, there exists a non-zero $\xi\in \z$ such that $\theta(\xi)\equiv 0$. We can thus decompose
\begin{equation}\label{unu}
\theta=aJ\xi^\flat+\theta_0, \qquad a\in\RM \quad \text{and}\quad \theta_0\perp \sp\{\xi,J\xi\}.
\end{equation}
For any $X\in\g$, we consider the function $f_X:=\langle X,J\xi\rangle$. By Cartan's formula (taking into account that $\xi\es\om=J\xi^\flat$, $\Lie_XJ\xi^\flat=0$, $\Lie_X\om=0$) we derive
\begin{equation*}
\begin{split}
df_X&=d(X\es J\xi^\flat)=-X\es dJ\xi^\flat+\Lie_XJ\xi^\flat=-X\es dJ\xi^\flat=-X\es d(\xi\es\om)\\
&=-X\es(-\xi\es d\om+\Lie_\xi\om)=X\es\xi\es d\om=X\es\xi\es (-2\theta\wedge\om)\\
&=X\es\big(-2\theta(\xi)\om+2\theta\wedge J\xi^\flat\big)\\
&=-2\theta(\xi)J\xi^\flat+2\theta(X)J\xi^\flat-2\theta\langle X,J\xi\rangle\\
&=2\langle aJ\xi,X\rangle J\xi^\flat+2\theta_0(X)J\xi^\flat-2a\langle X,J\xi\rangle J\xi^\flat-2\theta_0\langle X,J\xi\rangle\qquad \text{by} \, \eqref{unu}\, \text{and}\, \theta(\xi)=0\\
&=2\theta_0(X)J\xi^\flat-2\theta_0\langle X,J\xi\rangle=2\theta_0(X)J\xi^\flat-2\theta_0 f_X.
\end{split}
\end{equation*} 
Hence we have
\begin{equation}\label{doi}
df_X=2\theta_0(X)J\xi^\flat-2\theta_0f_X
\end{equation}
Then at any critical point $x$ of $f_X$, 
$$\theta_0\res_x(X_x)J\xi^\flat_x-2\theta_0\res_xf_X(x)=0.$$
But if non-zero,  $\theta_0\res_x$ and $J\xi^\flat_x$ are linearly independent, as $\theta_0(J\xi)=0$. As $\theta_0$ is $G$-invariant, it has constant norm, and hence if $\theta_0$ is not identically zero, \eqref{doi} implies $f_X(x)=0$ for all critical points $x$.  In particular, $f_X$ vanishes at its extremal points (which exist as $M$ is compact) and then $f_X$ identically vanishes. Now \eqref{doi} implies 
$$\theta_0(X)=0\qquad \text{ for all} \,\, X\in\g,$$
and thus $\theta_0\equiv 0$ so finally $J\theta^\sharp=-a\xi\in\z$.
\end{proof}

\begin{ecor}
$\dim\z\leq 2$.
\end{ecor}
\begin{proof}
Observe that, from the above arguments, $\theta$ defines a linear form on $\z$:
$$\z\ni X\mapsto \theta(X),$$ 
For any $\xi$ in the kernel of this linear form, Lemma \ref{jxi} and \eqref{unu} imply that $\theta=aJ\xi^\flat$ for some $a\in\RM$ and hence its kernel is one-dimensional.
\end{proof}

Form now on we assume
\begin{center}
\fbox{The group $G$ has center of dimension $2$: $\dim \z=2$}
\end{center}

Note that the scalar products on $M$ of elements in $\z$ are $G$-invariant, thus constant. We choose and fix a basis $\{\xi,\eta\}$ in $\z$ orthonormal at each point of $M$, such that $\xi=J\theta^\sharp$, $\eta\perp\xi$ and we decompose
$$\eta=b\theta^\sharp+\a, \qquad \a\perp \theta^\sharp.$$
As $\eta\perp J\theta^\sharp$, we also have $\a\perp J\theta^\sharp$.

\begin{ere} {\rm $b=\theta(\eta)\neq 0$, otherwise by Lemma \ref{lemap}, $\theta$ would vanish on the whole algebra of Killing fields on $M$, which is impossible.}
\end{ere}

The following result is standard but we include a proof for convenience (homogeneity implies regularity in other geometric structures too).

\begin{elem}
The orbits of $\xi$ are closed and of constant length ({\em i.e.} the foliation induced by $\xi$ is regular).
\end{elem} 
\begin{proof}
If the trajectory $\exp(t\xi)$ is not closed in $G$, then the closure of $\exp(\RM\xi)$ in $G$ is a $2$-dimensional torus, and hence equal with the identity component of the center $Z(G)$. It follows that any Killing field generated from $\z$ is perpendicular on $\theta^\sharp$. But we know already that $[\g,\g]\perp\theta^\sharp$ and hence $\theta^\sharp\perp\g$, contradiction. This shows that the subgroup $\exp(\RM\xi)\subset Z(G)$ is closed, thus isomorphic to $S^1$.

The orbits of $\xi$ on $M$ are thus closed circles. Moreover they have constant length because $G$ acts transitively on the orbit space.
\end{proof}

\begin{ere}{\rm A more restrictive result was already proven by I. Vaisman in \cite{vai}: on compact homogeneous Vaisman manifolds, the foliation generated by $\theta^\sharp$ and $J\theta^\sharp$ is regular and hence the manifold fibers in 1-dimensional complex tori over a compact homogeneous K\"ahler manifold.}
\end{ere}

From the previous lemma, $N:=M/\langle\xi\rangle$ is a $\mathcal{C}^\infty$ compact manifold. Moreover, the group $G/\exp(\RM\xi)$ acts on $N$ transitively, and hence $N$ is homogeneous.

Note that any basic ({\emph{i.e.} defined on $\xi^\perp\subset \T M$) and $\Lie_\xi$-invariant tensor on $M$ descends to a tensor on $N$. We claim that the  $\Lie_\xi$-invariant endomorphism $A:=\nabla \theta^\sharp$ (symmetric, as $d\theta=0$) is basic. Indeed, since $\xi$ is Killing of unit length, we have for every vector field $X$ on $M$:
$$g(\nabla_\xi\xi,X)=-g(\nabla_X,\xi,\xi)=-\tfrac12d(|\xi|^2)(X)=0,$$
whereas by \eqref{lck}
$$\nabla_\xi \om=\theta\wedge J\xi^\flat+J\theta\wedge\xi^\flat=\theta\wedge \theta+J\theta\wedge J\theta=0,$$
whence $\nabla_\xi J=0$, and thus:
$$A(\xi)=\nabla_\xi (J\xi)=(\nabla_\xi J)(\xi)+J(\nabla_\xi \xi)=0.$$

Let us consider the following tensor fields on $M$:
\begin{itemize}
\item The 1-form of unit length $\theta=J\xi^\flat$;
\item The symmetric tensor $g-J\theta\otimes J\theta=g-\xi^\flat\otimes\xi^\flat$;
\item The 2-form $\om-\theta\wedge J\theta=\om-\xi^\flat\wedge J\xi^\flat$;
\item The symmetric endomorphism $A=\nabla \theta^\sharp$;
\item The vector field $\a=\eta-b\theta^\sharp$.
\end{itemize}
All these tensors are basic by construction and $\Lie_\xi$-invariant since $\Lie_\xi$ preserves $g,\om,J,\theta$ and $\eta$ and commutes with $\sharp$ and $\flat$. 
We call $\theta_1$, $g_1$, $\om_1$, $A_1$ and $\a_1$ their projections on $N$. As $d\theta_1=0$, the distribution $\Ker(\theta_1)$ is integrable on $N$. Note that the leafs of $\Ker(\theta_1)$ are homogeneous and compact, as the group $G/Z(G)$ acts transitively on them.

Let $P$ be a fixed maximal integral leaf of $\Ker(\theta_1)$, and let $g_2$, $\om_2$, $\a_2$ be the restrictions of $g_1$, $\om_1$, $\a_1$ to $P$ (recall that $\theta_1(\alpha_1)=\theta(\alpha)=0$, so the vector field $\alpha_1$ is tangent to $P$).

\begin{epr}\label{holo}
The manifold $(P,g_2,\om_2)$ is K\"ahler homogeneous and the vector field $\a_2$ is holomorphic.
\end{epr}   
\begin{proof}
Let $\n^1$ and $\n^2$ be the Levi-Civita connections of $g_1$ (on $N$) and $g_2$ (on $P$)  respectively. Denote by $X,Y,Z...$ vector fields on $N$ {\em which are orthogonal to $\theta_1$}, identified with their restrictions to $P$ and with their horizontal lifts to $M$. Clearly both $\theta$ and $J\theta$ vanish identically on such vector fields on $M$.

Observe that $A_1$ is precisely the second fundamental form of the isometric immersion $P\hookrightarrow N$ whose unit normal is $\theta_1^\sharp$. We then have:
\begin{equation*}
\begin{split}
\n^2_XY&=\n^1_XY-g_1(\n^1_XY,\theta_1^\sharp)\theta_1^\sharp\\
&= \n^1_XY+g_1( A_1(X),Y)\theta_1^\sharp,
\end{split}
\end{equation*}
which immediately leads to
$$(\n^2_X\om_2)(Y,Z)=(\n^1_X\om_1)(Y,Z).$$
On the other hand,
$$g_1(\n^1_XY,Z)=g(\n_XY,Z),$$
and hence
$$(\n^1_X\om_1)(Y,Z)=\n_X(\om-\theta\wedge J\theta)(Y,Z).$$
As both $\theta$ and $J\theta$ vanish identically on $X,Y,Z\in\mathcal{X}(M)$, we obtain from \eqref{lck}:
$$\n_X(\om-\theta\wedge J\theta)(Y,Z)= \left(\theta\wedge JX^\flat+J\theta\wedge X^\flat-\n_X\theta\wedge J\theta-\theta\wedge\n_XJ\theta\right)(Y,Z)=0.$$
This implies $\n^2\om_2=0$ and thus $(P,g_2,\om_2)$ is K\"ahler.

By Lemma \ref{hv}, in order to verify that $\a_2$ is holomorphic, we need to prove that
$$\n^2_{J_2X}\a_2=J_2\n^2_X\a_2.$$
A direct computation shows that for $X$ and $Y$ as before:
\begin{equation*}
\begin{split}
g_2(\n^2_X\a_2,Y)&=g_1(\n^1_X\a_1,Y)=g(\n _X\a,Y)\\
&=-bg(\n _X\theta^\sharp,Y)+ g(\n _X\eta, Y),
\end{split}
\end{equation*}
and thus it is enough to prove the following two identities:
\begin{equation*}
\n_{JX}\theta^\sharp= J\n_X\theta^\sharp,\qquad
\n_{JX}\eta=J\n_X\eta.
\end{equation*}
The first relation follows directly from the fact that $\eta\in\g$ is Killing and preserves $\om$, thus is a holomorphic vector field on $M$: $\Lie_\eta J=0$. 
The same argument shows that $\xi$ is holomorphic, which by \eqref{int} implies $\Lie_{J\xi}J=0$ since $J$ is integrable. 
As $\theta^\sharp=J\xi$, this proves the second relation.
 \end{proof}

The next result is perhaps well known, but we provide a proof for the reader's convenience.

\begin{elem}\label{hvf}
Let $K$ be a Lie group acting effectively and transitively by holomorphic isometries on a compact K\"ahler manifold $(P,h,J)$. Then every $K$-invariant holomorphic vector field $\zeta$ on $P$ is parallel with respect to the Levi-Civita connection of $h$.
\end{elem}
\begin{proof}
The classification of compact homogeneous K\"ahler manifolds (see \cite{mat}) shows that up to a finite covering we can assume that $K=T^{2k}\times K_0$ and $P$ is isometric with a Riemannian product $T^{2k}\times (K_0/H_0)$, where $T^{2k}$ is a flat torus, $K_0$ is semi-simple and $\rk\, H_0 =\rk\, K_0$. Correspondingly, the tangent bundle of $P$ splits as a direct sum $\T P=\T _1P\oplus \T _2P$ of parallel $J$-invariant distributions. By Lemma \ref{hv}, the projections $\zeta_1$ and $\zeta_2$ of $\zeta$ on these distributions are both holomorphic vector fields on $P$. Of course, $\zeta_2$ is still $K_0$-invariant, so it has constant length. On the other hand, the Euler characteristic $\chi(G_0/H_0)$ equals the ratio of the cardinals of the Weyl groups of $K_0$ and $H_0$ (cf. \cite{adams}), so in particular it is non-vanishing. Consequently, a vector field of constant length on $G_0/H_0$ has to vanish. Thus $\zeta_2=0$.

Now, the restriction of $\zeta_1$ to any leaf $T^{2k}\times\{x_0\}$ is a holomorphic vector field on the flat torus. It is well known that on compact Ricci-flat K\"ahler manifolds, any holomorphic vector field is parallel (see {\em e.g.} \cite[Thm. 20.5]{kg}). Thus $\zeta_1$ is parallel in $\T _1$-directions. On the other hand, if $X_1$ and $X_2$ are any Killing vector fields on $P$ defined by some elements in the Lie algebras of $\t$ of $T^{2k}$ and $\k_0$ of $K_0$ respectively, then $X_1$ is clearly parallel on $P$ so 
$$0=\nabla_{X_2}X_1=\Lie_{X_2}X_1+\nabla_{X_1}X_2=\nabla_{X_1}X_2.$$
This shows that $\nabla_{X_1}X_2=0$ for every $X_1\in \T _1$ and for every $X_2\in \k_0$. In particular we get for every $X_2\in \k_0$:
$$0=\nabla_{\zeta_1}X_2=\Lie_{\zeta_1}X_2+\nabla_{X_2}\zeta_1=\nabla_{X_2}\zeta_1.$$
Since the vector fields in $\k_0$ span the distribution $\T _2$ at each point, this eventually shows that $\zeta=\zeta_1$ is parallel on $P$.
\end{proof} 

We are now in a position to prove our main result:

\begin{ath}\label{v}
A compact homogeneous locally conformally K\"ahler manifold $G/H$, with $\dim \z\geq 2$, is Vaisman.
\end{ath}
\begin{proof}

From Proposition \ref{holo}, we know that $\a_2$ is a $G/Z(G)$-invariant holomorphic vector field on the compact, K\"ahler manifold $P$ on which the group $G/Z(G)$ acts effectively and transitively by holomorphic isometries.  By Lemma \ref{hvf}, $\n^2\a_2=0$ and thus, for any $X,Y\perp \sp\{\theta^\sharp,J\theta^\sharp\}$, we obtain (with same type of computations as above) $\langle \n_X\a,Y\rangle=0$. From the definition of $\alpha$ we infer 
$$g(\n_X\eta,Y)=bg(\n_X\theta^\sharp,Y).$$
As the left hand side of the above identity is skew-symmetric in $X,Y$, while the right hand side is symmetric, both should vanish identically.
Since $b\neq 0$, this implies 
\begin{equation}\label{trei}
g(\n_X\theta^\sharp,Y)=0 \quad \text{for all}\,\,  X,Y\perp \sp\{\theta^\sharp,J\theta^\sharp\}. 
\end{equation}
Recall now that $\theta^\sharp$ is a holomorphic vector field (since $J\theta^\sharp\in\g$), so by Lemma  \ref{hv} we have $\n_{JX}\theta^\sharp=J\n_X\theta^\sharp$. This implies that the symmetric endomorphism $A:=\n\theta^\sharp$ commutes with $J$.

On the other hand, for every vector field $X$ on $M$ we have
$$g(A(\theta^\sharp),X)=g(A(X),\theta^\sharp)=g(\n_X{\theta^\sharp},\theta^\sharp)=\frac 12 d|\theta^\sharp|^2(X)=0,$$
which yields 
\begin{equation}\label{patru}
A(\theta^\sharp)=0\qquad \hbox{and thus}\quad A(J\theta^\sharp)=J(A(\theta^\sharp))=0.
\end{equation}
Equations \eqref{trei} and \eqref{patru} imply $A\equiv 0$ and hence $\n\theta=0$ as claimed.
\end{proof} 

%\begin{ere}{\rm Up to now, the only known  example of homogeneous LCK manifold is the (Vaisman) diagonal Hopf manifold described above, for which $\z=\sp\{\theta^\sharp, J\theta^\sharp\}$.
 
As a final remark, we recall that Kokarev  \cite{kk} introduced the class of {\em pluricanonical} LCK manifolds, characterized by the condition $(\n\theta)^{1,1}=0$ (which is weaker than the Vaisman condition $\n\theta=0$).
Very recently, P. Gauduchon proved \cite{paul} that compact homogeneous pluricanonical LCK manifolds are Vaisman, which, together with our Theorem \ref{v}, provides some further evidence in favor of the conjecture that {\em compact homogeneous LCK manifolds are Vaisman}. We note that this is the content of  Theorem 2 in \cite{haka},  but it seems that the proof therein is still incomplete.

\end{document}